\documentclass[a4paper, 12pt]{amsart}

\usepackage{amssymb}
\usepackage{amsthm}
\usepackage{amsmath}
\usepackage[mathscr]{eucal}
\usepackage{comment}

\theoremstyle{plain}
\newtheorem{thm}{Theorem}[section]		
\newtheorem{prop}[thm]{Proposition}

\newtheorem{lem}[thm]{Lemma}

\theoremstyle{definition}
\newtheorem{df}{Definition}[section]

\theoremstyle{remark}
\newtheorem{rmk}{Remark}[section]
\newtheorem*{ac}{Acknowledgements}

\newcommand{\zz}{\mathbb{Z}}
\newcommand{\qq}{\mathbb{Q}}
\newcommand{\rr}{\mathbb{R}}

\DeclareMathOperator{\card}{Card}

\DeclareMathOperator{\cl}{CL}

\DeclareMathOperator{\met}{M}

\DeclareMathOperator{\ult}{UM}

\DeclareMathOperator{\eee}{E}

\makeatletter
\@addtoreset{equation}{section}

\makeatother

\begin{document}

\title[Dense subsets in spaces of metrics]
{On dense subsets in spaces of metrics}
\author[Yoshito Ishiki]
{Yoshito Ishiki}
\address[Yoshito Ishiki]
{\endgraf
Graduate School of Pure and Applied Sciences
\endgraf
University of Tsukuba
\endgraf
Tennodai 1-1-1, Tsukuba, Ibaraki, 305-8571, Japan}
\email{ishiki@math.tsukuba.ac.jp}


\begin{abstract}
In spaces of metrics,  we investigate topological distributions 
of 
the doubling property, the uniform disconnectedness, and 
the uniform perfectness, which are 
the  quasi-symmetrically invariant properties appearing  in the David--Semmes theorem. 
We show that 
the set of all doubling metrics and  the set of  all uniformly disconnected metrics
are dense in  spaces of metrics 
on  finite-dimensional and zero-dimensional compact metrizable spaces, respectively. 
Conversely, this denseness of the sets  implies
the finite-dimensionality,  zero-dimensionality,  
and the compactness of metrizable spaces.  
We also determine the topological distribution of  the set of all uniformly perfect metrics  in the space of metrics on the Cantor set. 
\end{abstract}

\maketitle
\section{Introduction}

For a metrizable space $X$, 
we denote by $\met(X)$ 
the set of all metrics on $X$ that generate the same topology on  $X$. 
Define 
a metric   $\mathcal{D}_{X}: \met(X)^2\to [0, \infty]$ 
by 
$\mathcal{D}_X(d, e)=\sup_{x, y\in X}|d(x, y)-e(x, y)|$. 
In \cite{Ishiki2020int}, 
the author introduced the notion of a transmissible property 
which unifies  geometric properties defined by finite subsets of 
metric spaces, 
and  
proved that 
for every non-discrete  metrizable space $X$, 
the set  of all metrics in $\met(X)$ 
not satisfying a transmissible property 
with a  singular transmissible parameter 
is dense $G_{\delta}$  in $\met(X)$. 
Since the doubling property and the uniform disconnectedness are 
transmissible properties with singular parameters, 
the set  of all non-doubling  metrics 
and  
the set of all non-uniformly disconnected metrics are
dense $G_{\delta}$ in spaces of metrics (see \cite{Ishiki2020int}). 
In contrast to \cite{Ishiki2020int}, 
in this paper,  
we investigate topological distributions 
of 
the doubling property, 
the uniform disconnectedness, 
and  the uniform perfectness in spaces of metrics.

Niemytzki and Tychonoff \cite{NT1928} proved 
that 
a metrizable space $X$  is 
compact 
if and only if 
all metrics in $\met(X)$
are complete, 
and 
the author \cite{I4} proved 
an ultrametric analogue of their result 
(see \cite[Corollary 1.3]{I4}, see also \cite[Proposition 4.10]{DV2021}). 
Nomizu and Ozeki \cite{NO1961}  proved 
that 
a second countable differentiable manifold is compact 
if and only if 
all Riemannian metrics on the manifold are complete. 
As a development of these theorems, 
in the present  paper, 
we characterize 
the compactness and the finite-dimensionality or the $0$-dimensionality  
of a metrizable space   by 
the denseness of the set of all doubling metrics 
or 
all uniform disconnected metrics 
in the  space of metrics.

In this paper, 
a topological space is said to be \emph{finite-dimensional} (resp.~\emph{$0$-dimensional}) 
if its  covering dimension is finite  (resp.~$0$). 
The definition and 
 basic properties of the covering   dimension 
 can be seen in \cite{Dimtheo}. 
 
For a metric space $(X, d)$ 
and for a subset $A$ of $X$, 
we denote by $\delta_d(A)$ the diameter of $A$, 
and 
we define  $\alpha_d(A)=\inf\{\, d(x, y)\mid x\neq y,\ x, y\in A\, \}$. 
A metric space $(X, d)$ is said to be \emph{doubling} 
if 
there exist $\beta\in (0, \infty)$ and $C\in [1, \infty)$ such that 
for every finite subset $A$ of $X$ 
we have 
$\card(A)\le C\cdot(\delta_d(A)/\alpha_d(A))^{\beta}$, 
where 
the symbol ``$\card$'' stands for the cardinality. 
Let $X$ be a topological space. 
A subset $S$ is said to be $F_{\sigma}$ (resp.~$G_{\delta}$) 
if 
$S$ is 
the union  of  countably many  closed subsets of $X$ 
(resp.~the intersection of countably many open subsets of $X$).

\begin{thm}\label{thm:dense:D}
Let $X$ be a metrizable space. 
Then  the space $X$ is   finite-dimensional and  compact 
if and only if  
 the set of all doubling metrics in $\met(X)$ 
 is  dense $F_{\sigma}$ in 
 $(\met(X), \mathcal{D}_X)$. 
\end{thm}

A metric space is said to be \emph{uniformly disconnected} 
if 
there exists $\delta\in (0, 1)$ such that 
for every non-constant finite sequence $\{z_i\}_{i=1}^N$ in $X$
we have $\delta d(z_1, z_N)\le \max_{1\le i\le N}d(z_i, z_{i+1})$. 
This notion was introduced 
in  \cite{DS1997} in a different but equivalent way. 
Note that a metric space is uniformly disconnected 
if and only if 
the metric space can be  bi-Lipschitz embeddable 
into 
an ultrametric space 
(see \cite[Proposition 15.7]{DS1997}). 
Similarly to Theorem \ref{thm:dense:D}, 
we obtain: 
\begin{thm}\label{thm:denseUD}
Let $X$ be a metrizable space. 
Then  $X$ is  $0$-dimensional and  compact  
if and only if  
the set of all uniformly disconnected metrics in $\met(X)$ 
is dense $F_{\sigma}$ in 
$(\met(X), \mathcal{D}_X)$. 
\end{thm}

Let $X$ be a  set. 
A metric $d$ on $X$ is said to be an \emph{ultrametric} 
if 
for all $x, y, z\in X$ the metric $d$ satisfies the so-called 
strong triangle inequality:
$d(x, y)\le d(x, z)\lor d(z, y)$, 
where the symbol 
$\lor $ 
stands for the maximum operator on $\rr$. 
We say that a set 
$S$  
is a \emph{range set} 
if  
$S$ is a subset of $[0, \infty)$ 
and 
$0\in S$. 
For a range set 
$S$, 
we say that a metric 
$d$  
on 
$X$ 
is 
\emph{$S$-valued} 
if 
$d(X^2)$ 
is contained in 
$S$. 
For a range set $S$, 
and for a topological  space $X$, 
we denote by 
$\ult(X, S)$  
 the set of all $S$-valued ultrametrics  on $X$ 
 that generate 
the same topology on $X$. 
For a topological space $X$, 
and for a range set $S$, 
we define a function 
$\mathcal{UD}_X^S: \ult(X, S)^2\to [0, \infty]$
 by assigning $\mathcal{UD}_X^S(d, e)$ to  the infimum of 
 $\epsilon\in S\sqcup \{\infty\}$ such that 
 for all $x, y\in X$ we have 
 $d(x, y)\le e(x, y)\lor \epsilon$, 
 and 
 $e(x, y)\le d(x, y)\lor \epsilon$. 
The function 
$\mathcal{UD}_X^S$
 is an ultrametric on $\ult(X, S)$ valued in $\cl(S)\sqcup \{\infty\}$, 
 where $\cl(S)$ is  the closure of $S$ in $[0, \infty)$. 
 
 In \cite{I4},  
 the author investigate the topological distributions of 
 metrics not satisfying a transmissible property in 
 $(\ult(X, S), \mathcal{UD}_X^S)$.

We say that a range set $S$ has 
 \emph{countable coinitiality} if 
there exists 
a strictly decreasing sequence 
$\{r_i\}_{i\in \zz_{\ge0}}$ in $S$ convergent to $0$
as $i\to \infty$. 
As an ultrametric analogue of  Theorem \ref{thm:dense:D}, 
we obtain: 
\begin{thm}\label{thm:doubling:ult}
Let $S$ be a range set with the countable coinitiality. 
Let $X$ be an ultrametrizable space. 
Then  $X$ is compact if and only if 
the set of all doubling metrics in $\ult(X, S)$ 
is dense $F_{\sigma}$ in 
$(\ult(X, S), \mathcal{UD}_X^S)$. 
\end{thm}

\begin{rmk}
There are some results on relations between topological properties of 
metrizable spaces  and 
properties of spaces of ultrametrics. 
Let $X$ be an ultrametrizable space. 
Dovgoshey--Shcherbak \cite{DV2021} proved that 
 $X$ is compact if and only if 
all  $d\in \ult(X, [0, \infty))$ are totally bounded, 
and proved  that 
$X$ is 
separable if and only if 
for all $d\in \ult(X, [0, \infty))$ we have 
$\card(\{\, d(x, y)\mid x, y\in X\, \})\le \aleph_0$. 
\end{rmk}

Let $c\in (0, 1)$. 
A metric space $(X, d)$ is said to be 
\emph{$c$-uniformly perfect} 
if
for every $x\in X$, and 
for every $r\in (0, \delta_d(X))$, 
there exists $y\in X$ with $c\cdot r\le d(x, y)\le r$. 
A metric space is said to be \emph{uniformly perfect} if 
it is $c$-uniformly perfect for some $c\in (0, 1)$.  
We next investigate   topological distributions
of 
the set of all uniformly perfect metrics and its complement  
in the space of metrics on  the Cantor set. 
In this paper, 
let $\Gamma$ denote the Cantor set 
($\Gamma$ is also called the middle-third Cantor set).

\begin{thm}\label{thm:up}
The following statements  hold true:
\begin{enumerate}
\item 
The set of all uniformly perfect  metrics in $\met(\Gamma)$ 
is dense $F_{\sigma}$ in $(\met(\Gamma), \mathcal{D}_{\Gamma})$. 
\item  
The set of all non-uniformly perfect  metrics in $\met(\Gamma)$ 
is a dense $G_{\delta}$ set in $(\met(\Gamma), \mathcal{D}_{\Gamma})$.
\end{enumerate}
\end{thm}

We say that a range set $S$ is \emph{exponential}
if 
there exist $a\in (0, 1)$ and $M\in [1, \infty)$ such that
for every $n\in \zz_{\ge 0}$ 
we have $[M^{-1}a^n, Ma^n]\cap S\neq \emptyset$. 
As an ultrametric analogue of Theorem \ref{thm:up}, 
we obtain: 
\begin{thm}\label{thm:upult}
Let $S$ be a range set. 
Then the following hold true:
\begin{enumerate}
\item 
The set $S$ is exponential  
if and only if 
the set of all uniformly perfect  metrics in $\ult(\Gamma, S)$ 
is dense $F_{\sigma}$ in $(\ult(\Gamma, S), \mathcal{UD}_{\Gamma}^S)$. 
\item  
The set of all non-uniformly perfect  metrics in $\ult(\Gamma, S)$ 
is dense $G_{\delta}$ in $(\ult(\Gamma, S), \mathcal{UD}_{\Gamma}^S)$. 
\end{enumerate}
\end{thm}

Let $(X, d)$ and $(Y, e)$ be  metric spaces. 
A homeomorphism $f:X\to Y$ is said to be 
\emph{quasi-symmetric} 
if 
there exists  a homeomorphism 
$\eta:[0,\infty)\to [0,\infty)$ 
such that 
for all $x, y, z\in X$ 
and 
for every $t\in [0, \infty)$ 
the inequality  $d(x, y)\le td(x, z)$
implies  the inequality
$e(f(x), f(y))\le \eta(t)e(f(x), f(z))$. 
For example, 
all bi-Lipschitz homeomorphisms  are quasi-symmetric. 
Note that the doubling property, 
the uniform disconnectedness, 
and 
the uniform perfectness are
invariant under quasi-symmetric maps. 
David and Semmes \cite{DS1997} proved  
that 
if 
a compact metric space is  doubling,  
uniformly disconnected, 
and 
uniformly perfect,  
then  
 it is quasi-symmetrically equivalent 
 to the  Cantor set $\Gamma$ 
 equipped  with the Euclidean metric 
 (\cite[Proposition 15.11]{DS1997}). 
To simplify our description,
 the symbols 
 $\mathcal{P}_1$, 
 $\mathcal{P}_2$,  
 and $\mathcal{P}_3$ 
stand for the doubling property, 
the uniform disconnectedness, 
and 
the uniform perfectness, 
respectively. 
Before stating our results, 
for the sake of simplicity, 
we introduce the following notions:
\begin{df}
If a metric space $(X, d)$ satisfies 
a property 
$\mathcal{P}$, 
then we write 
$T_{\mathcal{P}}(X, d)=1$; 
otherwise, 
$T_{\mathcal{P}}(X, d)=0$.
For a triple $(u_1,u_2,u_3)\in \{0,1\}^3$, 
we say 
that 
a metric space $(X, d)$ 
is \emph{of type $(u_1,u_2,u_3)$} 
if we have
$T_{\mathcal{P}_k}(X, d)=u_k$ 
for all 
$k\in \{1, 2, 3\}$. 
\end{df}

A topological space is said to be a 
\emph{Cantor space} 
if 
it is 
homeomorphic to the Cantor set $\Gamma$. 
For a metric space $(X, d)$, 
we denote by 
$\mathcal{G}(X, d)$ 
the \emph{conformal gauge} of $(X, d)$ 
defined as 
the quasi-symmetric equivalent class of 
$(X, d)$, 
which 
is a basic concept in the conformal dimension theory 
(see e.g., \cite{MT2010}).
For each $(u,v,w)\in \{0,1\}^3$, 
we define 
\[
\mathscr{M}(u,v,w)
=
\{\,
\mathcal{G}(X,d)\mid \text{$(X,d)$ is a Cantor space of type $(u,v,w)$}
\,\}.
\]
\par

The  David--Semmes theorem mentioned above states 
that 
$\mathscr{M}(1,1,1)$ is a singleton. 
In contrast to this, 
 the author \cite{Ishiki2019} proved 
 that 
for every $(u,v,w)\in \{0,1\}^3$ except $(1,1,1)$, 
we have 
$\card(\mathscr{M}(u,v,w))=2^{\aleph_0}$ (see \cite[Theorem 2]{Ishiki2019}). 
As a development of this result, 
we investigate  topological distributions of metics 
of type $(u,v,w)$ for all 
$(u, v, w)\in \{0, 1\}^3$. 

\par
For every $(u, v, w)\in \{0, 1\}^3$, 
we put
\[
\eee(u, v, w)
=
	\{\,
	d\in \met(\Gamma)\mid \text{$(\Gamma, d)$ is of type $(u, v, w)$}
	\,\}.
\] 
Let $X$ be a topological space. 
A subset $M$ of $X$ is said to be 
$F_{\sigma\delta}$ 
(resp.~$G_{\delta\sigma}$) 
if 
$M$ is the intersection   of  countably many  
$F_{\sigma}$ subsets of $X$ 
(resp.~the union  of countably many $G_{\delta}$ subsets of $X$).  

\begin{thm}
The following three statement  hold true:
\begin{enumerate}\label{thm:type}
\item 
The set $\eee(1,1,1)$ 
is a dense $F_{\sigma}$ set of 
$(\met(\Gamma), \mathcal{D}_{\Gamma})$.
\item 
The set $\eee(0,0,0)$ 
is a dense $G_{\delta}$ 
subset of 
$(\met(\Gamma), \mathcal{D}_{\Gamma})$.

\item 
For every $(u, v, w)\in \{0, 1\}^{3}$ 
except $(1,1,1)$ and $(0,0,0)$, 
the set $\eee(u, v, w)$ is 
a dense $G_{\delta\sigma}$ and $F_{\sigma\delta}$ subset of 
$(\met(\Gamma), \mathcal{D}_{\Gamma})$.
\end{enumerate}
\end{thm}

\begin{ac}
The author 
would like to thank 
Professor Koichi Nagano for his advice and constant encouragement. 
\end{ac}

\section{The doubling property}

The following is known as 
the McShane--Whitney extension theorem. 
\begin{thm}\label{thm:mw}Let $l\in (0, \infty)$. 
Let $(X, d)$ be a metric space, 
and let $A$ be a subset of $X$. 
Then, 
for every $l$-Lipschitz map $f: A\to \rr$, 
there exists an $l$-Lipschitz map 
$F: X\to \rr$ 
such that 
$F|_A=f$. 
Moreover,  
 every $l$-Lipschitz map $f$ from
  $(A, d|_{A^2})$ into $\rr^n$ 
  with 
$\ell^{\infty}$-norm 
can be extended to 
an $l$-Lipschitz map from $(X, d)$ into $\rr^n$ 
with $\ell^{\infty}$-norm. 
\end{thm}

\begin{thm}\label{thm:denseIso}
Let $X$ be a metrizable compact  finite-dimensional space. 
Let $Q$ be the set of all metrics  $d\in \met(X)$ for which 
$(X, d)$ is isometrically  embeddable 
into a Euclidean space equipped with 
$\ell^{\infty}$-norm. 
Then  
the set $Q$ is dense in $(\met(X), \mathcal{D}_X)$. 
\end{thm}
\begin{proof}
In this proof, 
let $L_{N}$ denote the metric induced from the 
$\ell^{\infty}$-norm 
on $\rr^N$ for all $N\in \zz_{\ge 1}$. 
Put $n=\dim (X)$.  
Take arbitrary $d\in \met(X)$ 
and  
$\epsilon\in (0, \infty)$. 
Since $X$ is compact, 
there exists a finite sequence  $P=\{p_i\}_{i=1}^m$  
in 
$X$ 
such that 
$\alpha_d(P)\ge \epsilon$
and 
$X=\bigcup_{i=1}^mU(p_i, \epsilon)$. 
By the Kuratowski embedding theorem 
(see \cite{H2001}), 
there exists an isometric embedding from 
$(P, d|_{P^2})$ 
into 
$(\rr^m, L_m)$. 
Thus, 
by Theorem \ref{thm:mw},  
there exists a $1$-Lipschitz map 
$F: X\to \rr^m$ 
such that 
$F|_{P}$ is isometry. 
By \cite[Theorem 9.6]{Dimtheo}, 
there exists 
a topological embedding  $I: X\to \rr^{2n+1}$. 
Since 
$\rr^{2n+1}$ is 
homeomorphic to its open ball with radius  $\epsilon/2$, 
we may assume 
that 
$\delta_{L_{2n+1}}(I(X))<\epsilon$. 
Define 
$D\in \met(X)$ by 
$D(x, y)=L_{m}(F(x), F(y))\lor L_{2n+1}(I(x), I(y))$. 
Then 
the map 
$E:(X, D)\to  (\rr^{m+2n+1}, L_{m+2n+1})$ 
defined by 
$E(x)=(F(x), I(x))$ 
is an isometric embedding. 
Since $F|_{P}$ is an isometric embedding, 
and since 
$\alpha_{d}(P)\ge \epsilon$, 
we have $D(p_i, p_j)=d(p_i, p_j)$. 
Since $F$ is $1$-Lipschitz,  for every $x\in X$, 
the inequality $d(x, p_i)< \epsilon$ implies $D(x, p_i)< \epsilon$. 
For all  $x, y\in X$, 
take $p_i, p_j$ with $d(x, p_i)<\epsilon$ 
and $ d(y, p_j)<\epsilon$. 
Then, 
\begin{align*}
|D(x, y)-d(x, y)|\le d(x, p_i)+d(y, p_j)+D(x, p_i)+D(y, p_j)< 4\epsilon. 
\end{align*}
Thus,  
we conclude that $Q$ is dense in 
$(\met(X), \mathcal{D}_X)$. 
\end{proof}

The following is 
the Hausdorff metric  extension theorem \cite{Ha1930}:
\begin{thm}\label{thm:Haus}
Let $X$ be a metrizable space, 
and $A$ a closed subset of $X$. 
Then for every $d\in \met(A)$, 
there exists $D\in \met(X)$ with 
$D|_{A^2}=d$. 
\end{thm}

By Corollary 4.4 and Proposition 4.9 in \cite{Ishiki2020int}, 
we obtain: 
\begin{lem}\label{lem:doublingfsigma}
For a metrizable spaces $X$, 
the set of all doubling metrics in $\met(X)$ is $F_{\sigma}$
in $(\met(X), \mathcal{D}_X)$. 
\end{lem}

\begin{proof}[Proof of Theorem \ref{thm:dense:D}]
Let $T$ be the the set of all doubling metrics in $\met(X)$. 
Assume first 
that  
$X$ is a metrizable compact finite-dimensional space. 
Since all metric subspaces of the Euclidean spaces are doubling, 
$T$ contains 
the set  $Q$ stated in Theorem \ref{thm:denseIso}. 
Then, 
by Theorem \ref{thm:denseIso} and Lemma \ref{lem:doublingfsigma},  
the set $T$ is dense $F_{\sigma}$ in $\met(X)$. 
We next prove the opposite. 
If there exists a doubling metric in $\met(X)$, 
then, 
by the Assouad embedding theorem 
(see \cite[Theorem 12.1]{H2001}),  $X$ is finite-dimensional. 
For the sake of contradiction, 
suppose that $X$ is not compact. 
Then, 
there exists a countable  closed discrete subspace $F$ of $X$. 
Let $e$ be the metric on $F$ 
such that 
$e(x, y)=1$ whenever $x\neq y$. 
By Theorem \ref{thm:Haus}, 
there exists $D\in \met(X)$ with $D|_{F^2}=e$. 
Let $U$ be 
the  open ball centered at $D$ 
with radius $1/2$ in 
$(\met(X), \mathcal{D}_X)$. 
Take $d\in U$. 
Then,  
for every finite subset $A$ of $F$, 
we have 
$\delta_D(A)-1/2\le \delta_d(A)$ and 
$\alpha_d(A)\le \alpha_{D}(A)+1/2$. 
Since $1\le \alpha_D(A)$ and $1\le \delta_{D}(A)$, 
we have $\delta_D(A)/2\le \delta_d(A)$  
and 
$\alpha_d(A)\le 2\alpha_{D}(A)$. 
Since $D$ is not doubling, 
for every $C\in [1, \infty)$, 
and 
for every $\beta\in (0, \infty)$, 
there
exists a finite subset $B$ of $F$ 
with
$4C\cdot \left(\delta_D(B)/\alpha_D(B)\right)^{\beta}<\card(B)$, 
and hence we obtain 
$C\cdot \left(\delta_d(B)/\alpha_d(B)\right)^{\beta}<\card(B)$. 
Then  
$d$ is not doubling. 
Thus, 
the open set $U$ consists of non-doubling metrics, 
and 
hence the set $T$ is not dense in $\met(X)$. 
This finishes the proof. 
\end{proof}

\section{Amalgamation lemmas}

We provide new amalgamation lemmas of metrics and ultrametrics. 

\begin{prop}\label{prop:amal}
Let $I$ be a set. 
Let $(X, d)$ be a metric space. 
Let $\{B_i\}_{i\in I}$ be a covering  of $X$  
consisting of  mutually disjoint  clopen  subsets, 
and 
let $\{p_i\}_{i\in I}$ be points with $p_i\in B_i$. 
Let $\{e_i\}_{i\in I}$ be 
a set of metrics such that 
$e_i\in \met(B_i)$. 
Define a function $D:X^2\to [0, \infty)$ by 
\[
D(x, y)=
\begin{cases}
e_i(x, y) & \text{ if $x, y\in B_i$;}\\
e_i(x, p_i)+d(p_i, p_j)+e_j(p_j, y)  & \text{if $x\in B_i$ and $y\in B_j$.}
\end{cases}
\]
Then 
$D\in \met(X)$ 
and 
$D|_{B_i^2}=e_i$ 
for all $i\in I$. 
Moreover, 
if for every $i\in I$ 
we have 
$\delta_{d}(B_i)\le \epsilon$
 and 
 $\delta_{e_i}(B_i)\le \epsilon$, 
 then 
$\mathcal{D}_{X}(D, d)\le 4\epsilon$. 
\end{prop}
\begin{proof}
We first show 
that 
$D$ satisfies the triangle inequality. 
Take distinct $i, j, k\in I$, 
and take $x, y, z\in X$. 
In the case of $x, y\in B_i$ and $z\in B_j$,  
we have 
$D(x, y)=e_i(x, y)\le e_i(x, p_i)+e_i(p_i, y)\le D(x, z)+D(z, y)$. 
In the case of $x\in B_i$ and $y, z\in B_j$, 
we have 
\begin{align*}
D(x, y)&=e_i(x, p_i)+d(p_i, p_j)+e_j(p_j, y) \\
&\le e_i(x, p_i)+d(p_i, p_j)+e_j(p_j, z) +e_i(z, y)=
D(x, z)+D(z, y).
\end{align*}
In the case of $x\in B_i$,  $y\in B_j$ and $z\in B_k$,  
we have 
\begin{align*}
D(x, y)&=e_i(x, p_i)+d(p_i, p_j)+e_j(p_j, y) \\
&\le e_i(x, p_i)+d(p_i, p_k)+d(p_k, p_j)+e_j(p_j, z) +e_i(z, y)\\
&\le 
D(x, z)+D(z, y).
\end{align*}
Since $i, j, k$ and $x, y, z$ are arbitrary, 
we conclude that $D$ satisfies the triangle inequality. 
Since $\{B_i\}_{i\in I}$ is a disjoint family of clopen subsets, 
 $X$ is homeomorphic to 
 the disjoint union space induced from  
 $\{B_i\}_{i\in I}$. 
 Thus, 
 $D\in \met(X)$. 
 We next prove the latter part. 
 Take $x, y\in X$. 
 If $x, y\in B_i$, 
 then, 
by the assumption, 
we have $|D(x, y)-d(x, y)|\le 2\epsilon$.
If $x\in B_i$ and $ y\in B_j$ for some $i, j\in I$ with  $i\neq j$, 
then we have 
\begin{align*}
|D(x, y)-d(x, y)|&\le D(x, p_i)+D(p_j, y)+d(x, p_i)+d(y, p_j)\\
&\le  
\delta_d(B_i)+\delta_{e_i}(B_i)+\delta_d(B_j)+\delta_{e_j}(B_j)\le 4\epsilon.
\end{align*}
This completes the proof. 
\end{proof}

By replacing the symbol 
``$+$'' with ``$\lor$''  
in the proof of Proposition \ref{prop:amal}, 
we obtain the following proposition:
\begin{prop}\label{prop:amalult}
Let $I$ be a set. 
Let $S$ be a range set. 
Let $(X, d)$ be an $S$-valued  ultrametric space. 
Let $\{B_i\}_{i\in I}$ be a covering of $X$ 
consisting   of mutually disjoint  clopen  subsets, 
and let $\{p_i\}_{i\in I}$ be points with $p_i\in B_i$. 
Let $\{e_i\}_{i\in I}$ be a set of ultrametrics with 
$e_i\in \ult(B_i, S)$. 
Define a function  $D:X^2\to [0, \infty)$ by 
\[
D(x, y)=
\begin{cases}
e_i(x, y) & \text{ if $x, y\in B_i$;}\\
e_i(x, p_i)\lor d(p_i, p_j)\lor e_j(p_j, y)  & \text{if $x\in B_i$ and $y\in B_j$.}
\end{cases}
\]
Then $D\in \ult(X, S)$ and $D|_{B_i^2}=e_i$ for all $i\in I$. Moreover, 
if for every $i\in I$ 
we have 
$\delta_{d}(B_i)\le \epsilon$
and 
$\delta_{e_i}(B_i)\le \epsilon$, 
then 
$\mathcal{UD}_{X}^S(D, d)\le \epsilon$. 
\end{prop}
\begin{rmk}
Let $X$ be a metrizable space,  
let $A$ be a closed subset of $X$, 
and 
let $\{B_i\}_{i\in I}$ be 
a covering of mutually disjoint clopen subsets of $X$. 
The Hausdorff metric  extension theorem 
(Theorem \ref{thm:Haus}) states 
that
 a metric defined on the  squared set  $A^2$ 
 can be extended to a metric defined on $X^2$ 
 (see also Theorem \ref{thm:ultext}). 
 On the other hand, 
Proposition \ref{prop:amal} states 
that 
a metric defined on the set $\coprod_{i\in I}B_i^2$ 
can be extended to 
a metric defined on $X^2$. 
Note that the set  $\coprod_{i\in I}B_i^2$
is not a squared subset of $X^2$ in general. 
Dovgoshey--Martio--Vuorinen \cite{DMV2013} found 
the necessary and sufficient condition 
under which a weight of a weighted graph can be extended to 
a pseudometric on the vertex set  of the graph. 
Dovgoshey--Petrov \cite{DP2013} prove an ultrametric version of it. 
Propositions \ref{prop:amal} and \ref{prop:amalult} can be 
considered as a generalization of 
Dovgoshey--Martio--Vuorinen and Dovgoshey--Petrov's results. 
We also remark that
Proposition \ref{prop:amalult} is 
 a generalization of Dovgoshey--Shcherbak's construction of ultrametrics  (see the definition (4.11) in the proof of Theorem 4.7 in \cite{DV2021}). 
\end{rmk}

By the definition of the uniform perfectness, 
we obtain:
\begin{lem}\label{lem:upequiv}
A metric space $(X, d)$ is uniformly perfect 
if and only if 
there exist 
$c\in (0, 1)$ 
and 
$\delta\in (0, \infty)$ 
such that 
for every 
$x\in X$ 
and 
for every 
$r\in (0, \delta)$, 
there exists 
$y\in X$ 
with 
$cr\le d(x, y)\le r$. 
\end{lem}

Recall that 
the symbols
$\mathcal{P}_1$, 
$\mathcal{P}_2$,  
and 
$\mathcal{P}_3$ 
stand for the doubling property, 
the uniform disconnectedness, and 
the uniform perfectness, respectively. 

\begin{lem}\label{lem:k=123}
Let 
$(X, d)$ 
be a metric space. 
Let 
$\{B_i\}_{i=1}^n$ 
be a covering of $X$ 
consisting  of mutually disjoint clopen  subsets, 
and let 
$\{p_i\}_{i=1}^n$ 
be points with 
$p_i\in B_i$. 
Let 
$\{e_i\}_{i=1}^{n}$ 
be a set of metrics 
such that 
$e_i\in \met(B_i)$. 
Let  
$D$ 
be the metric constructed in Proposition \ref{prop:amal}
 from $d$ and $\{e_i\}_{i=1}^n$.  
Then for every $k\in \{1, 2, 3\}$, 
the following hold true:
\begin{enumerate}
\item If each $e_i$ satisfies $\mathcal{P}_k$, 
then so does $(X, D)$. 
\item If some $e_i$ does not satisfy $\mathcal{P}_k$, 
then neither does $(X, D)$. 
\end{enumerate}
\end{lem}
\begin{proof}
We first prove the statement (1). 
In the case of $k=1$, take a finite subset $A$ of $X$, 
and 
put $A_i=A\cap B_i$ for all $i\in \{1, \dots, n\}$. 
Since each $S_i$ is doubling, 
for each $i\in \{1, \dots, n\}$ 
there exist $C_i\in (0, \infty)$
and $\beta_i$ such that 
$\card(A_i)\le C_i(\delta_{e_i}(A_i)/\alpha_{e_i}(A_i))^{\beta_i}$. 
Put $C=\max_{1\le i\le n}{C_i}$ and 
$\beta=\max_{1\le i\le n}\beta_i$, 
then 
we have 
$\card(A)\le nC(\delta_{D}(A)/\alpha_{D}(A))^{\beta}$.
This proves the statement (1) for $k=1$. 
We next deal with   the case of $k=2$. 
By the assumption and by \cite[Proposition 15.7]{DS1997}, 
 for each $i\in \{1, \dots, n\}$
 there exist an ultrametric $w_i\in \met(B_i)$ and 
 a bi-Lipschitz map  $f_i: (B_i, e_i)\to (B_i, w_i)$. 
By applying 
 Proposition \ref{prop:amalult} to 
 $\{w_i\}_{i=1}^n$, we obtain 
  an ultrametric $R\in \met(X)$. 
  Define $f: (X, D)\to (X, R)$ by 
  $f(x)=f_i(x)$ if $x\in B_i$. 
  Then, $f$ is bi-Lipschitz, 
  and hence $(X, D)$ is uniformly disconnected. 
We next prove the case of $k=3$. 
Assume that all $B_i$ are $c$-uniformly perfect for some $c\in (0, 1)$. 
Put $m=\min_{1\le i\le n}\delta_{e_i}(B_i)$. 
Since each $B_i$ possesses at least two elements, 
we have $m>0$. 
Put $C=(1/2)\min\{c, m/ \delta_D(X)\}$. 
Then, 
$(X, D)$ is $C$-uniformly perfect. 
\par

We now prove the statement (2) in the lemma. 
Since the doubling property and uniform disconnectedness 
are hereditary to all metrics subspaces, 
the statements (2) for $k=1,  2$ are true. 
We now treat the case of $k=3$. 
Take  $c\in (0, 1)$. 
Put  $\delta=\min\{\, d(p_i, p_j)\mid i\neq j\, \}$. 
We may assume that  $e_1$ is   not uniformly perfect. 
Then, by Lemma \ref{lem:upequiv}, 
there exist $x\in B_1$ and $r\in (0, \delta)$ such that 
for all $y\in B_1$ we have 
$e_1(x, y)< cr$ or 
$r< e_1(x, y)$. 
By the definition of $D$, 
for all $y\in X$ we have 
$D(x, y)< cr$ 
or 
$r< D(x, y)$. 
Thus,  
$D$ is not uniformly perfect. 
\end{proof}

By Corollary 4.4 and Proposition 4.11 in  \cite{Ishiki2020int}, 
we obtain: 
\begin{lem}\label{lem:udfsigma}
For a metrizable spaces $X$, 
the set of all uniformly disconnected  metrics in $\met(X)$ 
is $F_{\sigma}$
in $(\met(X), \mathcal{D}_X)$. 
\end{lem}

\begin{proof}[Proof of Theorem \ref{thm:denseUD}]
Assume  that 
$X$ is compact and $0$-dimensional. 
Take arbitrary 
$d\in \met(X)$ and $\epsilon\in (0, \infty)$. 
Since $X$ is compact and $0$-dimensional, 
there exists a covering 
$\{B_i\}_{i=1}^n$ of $X$  of  mutually disjoint
clopen subsets with 
$\delta_{d}(B_i)\le \epsilon$. 
Since each $B_i$ is $0$-dimensional, 
there exists a uniformly disconnected metric  
$e_i\in \met(B_i)$ 
with 
$\delta_{e_i}(B_i)\le \epsilon$. 
Applying Proposition \ref{prop:amal} 
to $d$ and $\{e_i\}_{i=1}^n$, 
we obtain $D\in \met(X)$ 
with $\mathcal{D}_X(D, d)\le 4\epsilon$. 
By  Lemma \ref{lem:k=123}, 
the metric $D$ is uniformly disconnected. 
Thus, 
by  Lemma \ref{lem:udfsigma}, 
the set of all uniformly disconnected  metrics 
is dense $F_{\sigma}$ in $\met(X)$. 
We prove the opposite. 
If there exists a uniformly disconnected  metric in $\met(X)$, 
then $X$ is $0$-dimensional. 
For the sake of contradiction, 
suppose that $X$ is not compact. 
Then, there exists a countable  closed discrete subset $F$ of $X$. 
Identify $F$ with $\zz$, 
and 
let $e$ be the relative Euclidean  metric  on 
$\zz (=F)$. 
By Theorem \ref{thm:Haus}, 
there exists  $D\in \met(X)$ such that 
$D|_{F^2}=e$. 
Let $U$ be the open ball centered at $D$ with radius $1/2$ in 
$(\met(X), \mathcal{D}_X)$
Take $d\in U$. 
Since $e$ is not uniformly disconnected, 
for every $\delta\in (0, 1)$ there exists a non-constant 
finite sequence  
$\{z_i\}_{i=1}^N$ in $F$ 
with  
$\max_{1\le i\le N}D(z_{i}, z_{i+1})<4\delta D(z_1, z_N)$. 
By   $\mathcal{D}_X(d, D)<1/2$, and by 
$1\le D(z_1, z_{Z})$ and $1\le \max_{1\le i\le N}D(z_{i}, z_{i+1})$, 
we have 
$D(z_1, z_{N})\le 2d(z_1, z_N)$ and 
$\max_{1\le i\le N}d(z_{i}, z_{i+1})\le 2\max_{1\le i\le N}D(z_{i}, z_{i+1})$. 
Then we obtain 
$\max_{1\le i\le N}d(z_{i}, z_{i+1})<\delta d(z_1, z_{N})$. 
This implies that $d$ is not  uniformly disconnected. 
Thus, 
the open subset 
$U$ consists  of non-uniformly disconnected metrics, 
and hence the set of all uniformly disconnected metrics  is not dense 
in $\met(X)$. 
This finishes the proof. 
\end{proof}

Similarly to Lemma \ref{lem:k=123}, 
we obtain:
\begin{lem}\label{lem:k=123ult}
Let $S$ be a range set. 
Let $(X, d)$ be 
an $S$-valued ultrametric space. 
Let $\{B_i\}_{i=1}^n$ be a covering of $X$ 
consisting   of mutually disjoint  clopen  subsets, 
and let $\{p_i\}_{i=1}^n$ be points 
with 
$p_i\in B_i$. 
Let $\{e_i\}_{i=1}^{n}$ 
be  metrics such that 
$e_i\in \ult(B_i, S)$. 
Let  $D$ be the  metric constructed 
in  Proposition \ref{prop:amalult} 
from $d$ and $\{e_i\}_{i=1}^n$. 
Then for all $k\in \{1, 3\}$, 
the following hold true:
\begin{enumerate}
\item If each $e_i$ satisfies $\mathcal{P}_k$, 
then so does $(X, D)$. 
\item If some $e_i$ does not satisfy $\mathcal{P}_k$, 
then neither does $(X, D)$. 
\end{enumerate}
\end{lem}

The author \cite{I4} proved the extension theorem on ultrametrics, 
which is 
an ultrametric analogue of Theorem \ref{thm:Haus} 
(see \cite[Theorem 1.2]{I4}). 
\begin{thm}\label{thm:ultext}
Let $S$ be a range set. 
Let $X$ be a topological space with 
$\ult(X, S)\neq \emptyset$, 
and let $A$ be a closed subset of $X$. 
Then for every $d\in \ult(A, S)$
there exists $D\in \ult(X, S)$ with 
$D|_{A^2}=d$. 
\end{thm}

By Corollary 6.4 and Proposition 6.8 in \cite{I4}, 
we obtain: 
\begin{lem}\label{lem:ultdoublingfsigma}
Let $S$ be a range set. 
For a topological space $X$, 
the set of all doubling  metrics in 
$\ult(X, S)$ is $F_{\sigma}$
in $\ult(X, S)$. 
\end{lem}

\begin{proof}[Proof of Theorem \ref{thm:doubling:ult}]
By Lemma \ref{lem:ultdoublingfsigma}, 
the set of all doubling metrics in $\ult(X, S)$ 
is  $F_{\sigma}$ in $\ult(X, S)$. 
Assume first   that $X$ is compact. 
By \cite[Proposition 2.12]{I4}, 
we have $\ult(X, S)\neq \emptyset$. 
Take arbitrary $d\in \ult(X, S)$ and $\epsilon\in (0, \infty)$. 
Since $X$ is compact  and  $0$-dimensional, 
there exists a disjoint  covering  $\{B_i\}_{i=1}^n$  of 
 clopen subsets with $\delta_{d}(B_i)\le \epsilon$. 
Then, 
there exists a doubling metric  $e_i\in \ult(B_i, S)$ with 
$\delta_{e_i}(B_i)\le \epsilon$. 
Applying Proposition \ref{prop:amalult} 
to $d$ and $\{e_i\}_{i=1}^n$, 
we obtain $D\in \ult(X, S)$ 
with $\mathcal{UD}_X^S(D, d)\le \epsilon$. 
By  Lemma \ref{lem:k=123ult}, 
the metric $D$ is doubling. 
Thus, 
the set of all doubling metrics in $\ult(X, S)$ 
is dense in $\ult(X, S)$. 
We next prove the opposite. 
Similarly to the proof of Theorem 
\ref{thm:dense:D}, 
by using  Theorem \ref{thm:ultext}, 
we conclude that if the set of all doubling metrics in $\ult(X, S)$ 
is dense in 
$\ult(X, S)$, 
then $X$ is compact. 
\end{proof}

\section{The uniform perfectness}

Fix a countable dense subset $P$ of $\Gamma$. 
For $c\in (0, 1)$, 
let $K(c)$ denote the set of all 
$d\in \met(\Gamma)$ such that 
for every $r\in (0, 1)\cap \qq$, 
and for every 
$x\in P$, there exists $y\in P$ satisfying that $c\cdot r\le d(x, y)\le r$. 
Let $K$ denote the set of all uniformly perfect metrics in $\met(\Gamma)$. 

\begin{lem}\label{lem:upfsigma}
The set $K$ is an 
$F_{\sigma}$ subset of $\met(\Gamma)$. 
\end{lem}
\begin{proof}
By the definitions, 
we have  
$K=\bigcup_{c\in (0, 1)\cap \qq}K(c)$. 
For every $c\in (0, 1)\cap \qq$, 
we  prove that 
 $\cl(K(c))$ is contained in $K(c/4)$, 
 where $\cl$ is the closure operator of 
 $\met(\Gamma)$. 
 Take $d\in\cl(K(c))$, 
 and take a sequence $\{d_n\}_{n\in \zz_{\ge 0}}$ in $K(c)$ 
such that $d_n\to d$ as $n\to \infty$ 
in $(\met(\Gamma), \mathcal{D}_{\Gamma})$. 
Then for every $n\in \zz_{\ge 0}$,  
for every $r\in (0, 1)\cap \qq$,  
and for every $x\in P$, 
there exists $y(n, r, x)\in P$ such that 
$c\cdot (r/2)\le d_n(x, y(n, r, x))\le r/2$.
Since $(0, 1)\cap \qq$ and $P$ are countable, 
and since $\Gamma$ is compact,  
we can apply  Cantor's diagonal argument to $y(n, r, x)$, 
and hence 
we obtain a strictly increasing map $\phi:\zz_{\ge 0}\to \zz_{\ge 0}$ 
such that 
for every $r\in (0, 1)\cap \qq$,  
and for every $x\in P$, 
the sequence $\{y(\phi(n), r, x)\}_{n\in \zz_{\ge 0}}$ converges to 
a point in $X$, say $z(r, x)$. 
Take  $p(r, x)\in P$ with  
$d(p(r, x), z(r, x))\le (c/4)r$. 
Letting $n \to \infty$, 
we have 
$c\cdot (r/2)\le d(x, z(r, x))\le r/2$, 
and hence 
we have $(c/4)\cdot r\le d(x, p(r, x))\le r$. 
Thus,  $d\in K(c/4)$, 
and hence 
 $\cl(K(c))$ is contained in $K(c/4)$. 
By this observation, 
we conclude that
$K=\bigcup_{c\in (0, 1)\cap \qq}\cl(K(c))$. 
Thus $K$ is  $F_{\sigma}$ in $\met(\Gamma)$. 
\end{proof}
Similarly to Lemma \ref{lem:upfsigma}, we obtain:
\begin{lem}\label{lem:upult}
Let $S$ be a range set. 
The set $K\cap \ult(\Gamma, S)$ is 
$F_{\sigma}$ in  $\ult(\Gamma, S)$. 
\end{lem}

\begin{proof}[Proof of Theorem \ref{thm:up}]
By Lemma \ref{lem:upfsigma}, 
it suffices to show that $K$ is dense in 
$\met(\Gamma)$. 
Take arbitrary $d\in \met(\Gamma)$ 
and 
$\epsilon \in (0, \infty)$. 
Since $\Gamma$ is $0$-dimensional and compact, 
there exists a  covering 
$\{B_i\}_{i=1}^n$   of 
mutually disjoint 
clopen non-empty subsets  with $\delta_{d}(B_i)\le \epsilon$. 
Note that each $B_i$ is a Cantor space 
(see \cite[Corollary 30.4]{W1970}). 
Identify $B_i$ and $\Gamma$, 
and let $e_i\in \met(B_i)$ be the identified metric
with $\epsilon\cdot E$, 
where $E$ is the relative Euclidean metric on $\Gamma$. 
Then
each $e_i$ is 
  uniformly perfect and satisfies 
$\delta_{e_i}(B_i)\le \epsilon$. 
Applying Proposition \ref{prop:amal} 
to $d$ and  $\{e_i\}_{i=1}^{n}$, 
we obtain $D\in \met(\Gamma)$ 
with $\mathcal{D}_{\Gamma}(d, D)\le 4\epsilon$. 
Lemma \ref{lem:k=123} implies that 
$D$ is uniformly perfect. 
Thus $K$ is dense in $\met(\Gamma)$. 
This finishes the proof. 
\end{proof}

\begin{lem}\label{lem:conE}
Let $S$ be an exponential  range set. 
Then there exists a strictly decreasing sequence 
$\{s(n)\}_{n\in \zz_{\ge 0}}$ in $S$
satisfying that 
there exists $a\in (0, 1)$ and $M\in [1, \infty)$ 
such that 
for every $n\in \zz_{\ge 0}$ 
we have $M^{-1}a^n\le s(n)\le Ma^n$. 
\end{lem}
\begin{proof}
By the assumption on $S$, 
there exist $b\in (0, 1)$ and $M\in [1, \infty)$ 
such that 
for every $n\in \zz_{\ge 0}$, 
we have $[M^{-1}b^n, Mb^n]\cap S\neq \emptyset$. 
Put $p=-\log M/\log b$. Put $a=b^{2p+1}$.
Then, 
$[M^{-1}a^n, Ma^n]\cap S\neq \emptyset$ 
and 
$Ma^{n+1}<M^{-1}a^n$ 
for all $n\in \zz_{\ge 0}$. 
This leads to the lemma. 
\end{proof}

A sequence 
$s: \zz_{\ge 0}\to (0, \infty)$ said to be \emph{shrinking} 
if 
it is a strictly decreasing sequence $(0, \infty)$ convergent to $0$. 
For a shrinking sequence $s: \zz_{\ge 0}\to (0, \infty)$, 
and for every $m\in \zz_{\ge 0}$,  
define  $s^{\{m\}}:\zz_{\ge 0}\to (0, \infty)$ by 
$s^{\{m\}}(n)=s(m+n)$. 
Then $s^{\{m\}}$ is shrinking. 
\par

Let $2^{\omega}$ be the set of all maps from 
$\zz_{\ge0}$ into $\{0, 1\}$. 
Define  
a valuation $v: 2^{\omega}\times 2^{\omega}\to [0, \infty]$ 
by 
$v(x, y)=\min\{\, n\in \zz_{\ge 0}\mid x(n)\neq y(n)\, \}$ if $x\neq y$; 
otherwise $v(x, y)=\infty$. 
Let $s$ be a shrinking sequence. 
Put $s(\infty)=0$. 
Define a metric $d_{s}$ on $2^{\omega}$ by 
$d_{s}(x, y)=s(v(x, y))$. 
Then 
for every shrinking sequence $s: \zz_{\ge 0}\to (0, \infty)$, 
and  
for every $m\in \zz_{\ge 0}$, 
the metric space $(2^{\omega}, d_{s^{\{m\}}})$ 
is a Cantor space. 
Note that for every $x\in 2^{\omega}$, 
and for every $n\in \zz_{\ge 0}$, 
there exists $d_{s}(x, y)=s(n)$. 
The metric space $(2^{\omega}, d_s)$ 
is called a \emph{sequentially metrized 
 Cantor space} in the author's paper  \cite{Ishiki2019}, 
 and the author investigated 
 the doubling property, the uniform disconnectedness, 
 and 
the uniform perfectness 
of sequentially metrized 
Cantor spaces. 
The following lemma is essentially contained in  
 \cite[Lemma 6.4]{Ishiki2019}. 
 
\begin{lem}\label{lem:sequp}
Let $S$ be a range set. 
Let $s:\zz_{\ge 0}\to S$ be a shrinking sequence in $S$. 
If there exist $a\in (0, 1)$ and $M\in [1, \infty)$  
 such that 
$M^{-1}a^n\le s(n)\le Ma^n$, 
then 
for every $m\in \zz_{\ge 0}$, 
the metric space $(2^{\omega}, d_{s^{\{m\}}})$ is 
uniformly perfect. 
\end{lem}
\begin{proof}
It suffices to show the case of $m=0$. 
Put $c=M^{-2}a$. 
Take arbitrary  $x\in 2^{\omega}$. 
For every $r\in (0, \infty)$, 
take $n\in \zz_{\ge 0}$ such that 
$s(n+1)< r\le s(n)$. 
Take $y\in 2^{\omega}$ with $d(x, y)=s(n+1)$.
Then, 
$cr\le d(x, y)\le r$. 
This implies 
that 
$(2^{\omega}, d_{s})$  is uniformly perfect. 
\end{proof}

\begin{proof}[Proof of Theorem \ref{thm:upult}]
 Similarly to  Theorem \ref{thm:doubling:ult}, 
the statement (2) follows from Proposition \ref{prop:amalult} and 
Lemmas \ref{lem:k=123ult} and \ref{lem:upult}. 
\par

We now prove the statement (1). 
Assume first that $S$ is exponential. 
Then by Lemma \ref{lem:conE},  
there exist $a\in (0, 1)$ 
and $M\in [1, \infty)$ 
and  a strictly decreasing  sequence $\{s(n)\}_{n\in \zz_{\ge 0}}$ 
 in $S$ 
 such that 
 $M^{-1}a^n\le s(n)\le Ma^n$. 
 Take arbitrary $d\in \ult(\Gamma, S)$ 
 and $\epsilon\in (0, \infty)$. 
Then, 
 there exists 
 a covering  $\{B_i\}_{i=1}^N$ of $\Gamma$ 
 consisting of mutually disjoint non-empty clopen subsets of with 
 $\delta_d(B_i)\le \epsilon$. 
 Note that each $B_i$ is a Cantor space. 
 For a  sufficiently large $m\in \zz_{\ge 0}$, 
 the space 
 $(2^{\omega}, d_{s^{\{m\}}})$ 
 satisfies $\delta_{d_{s^{\{m\}}}}(2^{\omega})\le \epsilon$. 
Identify $B_i$ and $2^{\omega}$, 
and let $e_i\in \ult(B_i, S)$ be the  identified metric with $d_{s^{\{m\}}}$. 
By applying Proposition \ref{prop:amalult} to $\{e_i\}_{i=1}^N$ and $d$, 
we obtain an $S$-valued ultrametric $D$ in $\ult(\Gamma, S)$ 
with $\mathcal{UD}_{X}^S(D, d)\le \epsilon$. 
By Lemmas \ref{lem:k=123} and \ref{lem:sequp}, 
the metric $D$ is uniformly perfect. 
This leads to the former part of the statement (1). 
Assume next that 
$S$ is not exponential. 
We now prove that every $d\in \ult(\Gamma, S)$ 
is not uniformly perfect. 
Take arbitrary $c\in (0, 1)$. 
Since $S$ is not exponential, 
there exists $n\in \zz_{\ge 0}$ 
with $[c^{n+1}, c^{n-1}]\cap S=\emptyset$. 
Take any  $x\in \Gamma$. 
Put $r=c^{n-1}$. 
Then every $y\in \Gamma$ satisfies 
$d(x, y)\le cr$ or $r\le d(x, y)$. 
Thus,  
every $d\in \ult(\Gamma, S)$ is not uniformly perfect. 
This finishes the proof. 
\end{proof}

Before proving Theorem \ref{thm:type}, 
remark that each  $\eee(u, v, w)$ contains a metric with 
arbitrary small diameter. This follows from 
the facts that   $\eee(u, v, w)\neq \emptyset$ (see \cite[Theorem 1.2]{Ishiki2019}), 
 and that  
for every $(u, v, w)\in \{0, 1\}^3$ and 
for every  $\epsilon\in (0, \infty)$, 
if  $d\in\eee(u, v, w)$, 
then $\epsilon d\in \eee(u, v, w)$. 

\begin{proof}[Proof of Theorem \ref{thm:type}]
Fix $(u, v, w)\in \{0, 1\}^3$. 
We now prove 
that 
$\eee(u, v, w)$ is dense in $\met(\Gamma)$. 
Take arbitrary $d\in \met(\Gamma)$ and $\epsilon \in (0, \infty)$. 
Since $\Gamma$ is compact and $0$-dimensional, 
there exists a covering  $\{B_i\}_{i=1}^n$ of $\Gamma$
consisting of 
mutually disjoint non-empty clopen subsets
 with $\delta_{d}(B_i)\le \epsilon$. 
Since the set $\eee(u, v, w)$ contains a metric with 
arbitrary small diameter, 
there exists  $e_i\in \met(B_i)$ of type $(u, v, w)$  with 
$\delta_{e_i}(B_i)\le \epsilon$. 
Applying Proposition \ref{prop:amal} to $d$ and $\{e_i\}_{i=1}^n$, 
we obtain $D\in \met(\Gamma)$ 
with $\mathcal{D}_{\Gamma}(D, d)\le 4\epsilon$. 
By  Lemma \ref{lem:k=123}, the metric $D$ is of type $(u, v, w)$. 
Thus, 
$\eee(u, v, w)$ is dense in $\met(\Gamma)$. 
For each $k\in \{1, 2, 3\}$, 
let $W(k, 1)$ denote the set of all metrics satisfying the property $\mathcal{P}_k$ in $\met(\Gamma)$, 
and put $W(k, 0)=\met(\Gamma)\setminus W(k, 1)$. 
By Lemmas \ref{lem:doublingfsigma}, \ref{lem:udfsigma}, 
and \ref{lem:upfsigma}, 
for all $k\in \{1, 2, 3\}$, 
the sets $W(k, 1)$ and $W(k, 0)$ are  $F_{\sigma}$ and $G_{\delta}$ in $\met(\Gamma)$, respectively. Thus, 
for all $(u, v, w)\in \{0, 1\}^3$, 
we have 
\begin{equation}\label{eq:eee}
\eee(u, v, w)=W(1, u)\cap W(2, v)\cap W(3, w).
\end{equation} 
By the equality (\ref{eq:eee}), 
the sets  $\eee(1,1,1)$ and $\eee(0, 0, 0)$ are
$F_{\sigma}$ and $G_{\delta}$, respectively. 
If $(u, v, w)\in \{0, 1\}^3$ coincides 
with neither  $(0,0,0)$ nor $(1,1,1)$, 
then
by the equality (\ref{eq:eee}), 
the set $\eee(u, v, w)$  is   the intersection of an $F_{\sigma}$ set 
and  a $G_{\delta}$ set in $\met(\Gamma)$. 
 Since $\met(\Gamma)$ is a  metrizable space, 
 the set  $\eee(u, v, w)$ is 
$F_{\sigma\delta}$ and $G_{\delta\sigma}$ in $\met(\Gamma)$.
 This finishes the proof. 
\end{proof}

\bibliographystyle{amsplain}
\bibliography{bibtex/cptmet.bib}

\end{document}